\documentclass[12pt]{article}
\usepackage{graphics,color}
\usepackage{graphicx}
\usepackage{amsfonts}
\usepackage{amssymb}
\usepackage{amsmath}   
\usepackage{amsthm}
\usepackage{latexsym}
\usepackage{subfigure}

\newtheorem{thm}{Theorem}

\newtheorem{rek}[thm]{Remark}
\numberwithin{equation}{section}

\newtheorem{lem}{Lemma}[section]
\newtheorem{prop}{Proposition}[section]

\newcommand\be{\begin{equation}}
\newcommand\ee{\end{equation}}
\newcommand\bea{\begin{eqnarray}}
\newcommand\eea{\end{eqnarray}}
\newcommand\bi{\begin{itemize}}
\newcommand\ei{\end{itemize}}
\newcommand\ben{\begin{enumerate}}
\newcommand\een{\end{enumerate}}
\newcommand\bc{\begin{center}}
\newcommand\ec{\end{center}}
\newcommand\ba{\begin{array}}
\newcommand\ea{\end{array}}

\newcommand{\Keywords}[1]{\par\noindent {\small{\em Keywords\/}: #1}}

\newcommand{\AMS}[1]{\par\noindent {\small{\em AMS subject classification\/}: #1}}

\begin{document}
\title{Stability for second-order chaotic sigma-delta quantization}

\author{Lauren Bandklayder\thanks{L. Bandklayder is with the Department
of Mathematics, New York University, e-mail: lab505@nyu.edu.} %  
\hspace{2mm}
        and~Rachel~Ward%
\thanks{R. Ward is with the Department of Mathematics, University of Texas at Austin, rward@math.utexas.edu.}% <-this %
}

\maketitle

\begin{abstract}
%\boldmath
We prove that second-order (double-loop) chaotic sigma-delta schemes are \emph{stable}: within a certain parameter range, all state variables of the system are guaranteed to remain uniformly bounded. To our knowledge this is the first general stability result for chaotic sigma-delta schemes of order greater than one.  Invariably as the amount of expansion added to the system is increased, the dynamic range of the input must decrease for stability to be guaranteed.   We give explicit bounds on this trade-off and verify through numerical simulation that these bounds are near-optimal.
\end{abstract}
\begin{center}
\AMS{94C99, 94A12, 93D99, 37N35}
\end{center}

\section{Introduction}
Analog-to-digital conversion is the process of representing a continuous signal by a bitstream, or a sequence consisting of $+1 $'s and $-1$'s.  
Sigma-delta quantization schemes often preferred in practice for their simplicity and robustness.  The standard algorithms however tend to produce periodic output, causing audible idle tones.   One approach towards breaking up such periodicities is to modify the parameters in the standard sigma-delta scheme and implement instead \emph{chaotic} sigma-delta quantization.  However, rigorous results about the stability of the double-loop chaotic sigma-delta scheme have remained elusive.  It is our goal in this paper to prove that under the explicit parameter conditions which we will derive, all state variables in chaotic sigma-delta recursions are guaranteed to remain bounded. 

 {\bf\normalsize Background on analog-to-digital conversion.}
 Many signals (e.g. sound, light, and pressure) are naturally produced as bandlimited signals, being well-approximated locally as linear combinations of sines and cosines not exceeding a maximal frequency, or bandwidth.  As these signals are continuous, they must be sampled and quantized in order to be stored in digital format.  
 Recall that a bounded real-valued function $f$ is bandlimited and has bandwidth $\Omega$ if its Fourier transform $\widehat{f}(\omega) = \int_{-\infty}^{\infty} f(t) \exp{(-2\pi i \omega t)} dt$
 vanishes outside the bounded interval $[-\Omega, \Omega]$. The classic sampling theorem gives rigorous bounds on the sampling rate neessary for accurate reconstruction; namely, it states that a bandlimited function can be recovered from its values at evenly-spaced sampling points if the sampling frequency is at least twice the bandwidth.  For functions $f$ with bandwidth $\Omega = 1/2$ (normalized for simplicity), the sampling theorem states that for any $T > 1$,
\begin{equation}
f(t) = \frac{1}{T} \sum_{n \in \mathbb{Z}} f \Big( \frac{n}{T } \Big) g \Big( t - \frac{n}{T} \Big),
\label{shannonsamp}
\end{equation}  
where $g$ is any function (a.k.a. low-pass filter) whose Fourier transform is unity inside the interval $[-1, 1]$ and zero outside of an interval $[-T_0, T_0]$ with $T_0 \leq T$ \cite{DD03}. 

For a bandlimited function, then, analog-to-digital conversion then amounts to mapping the sequence of samples $f_n = f(\frac{n}{T})$ to a sequence of discrete values $q_n$ in such a way that the original function can be reconstructed to within a small error at a later time via  
\begin{equation}
\label{ftilde}
\widetilde{f}(t) = \frac{1}{T} \sum_{n \in \mathbb{Z}} q_n g \Big( t - \frac{n}{T} \Big) 
\end{equation}  
In pulse code modulation schemes, $q_n$ is taken to be the binary expansion of $f_n$, truncated to its first $K$ bits.  sigma-delta $(\Sigma \Delta)$ quantization is an alternative quantization procedure, popular in practice for its robustness and ease of implementation.  For a given input sequence, a $K$th-order $\Sigma \Delta$ scheme dynamically generates a sequence of discrete output in such a way that at each step, the quantization error $f_n - q_n$ equals the $K$th-order difference of a bounded state variable $v = (v_n)_{n \in \mathbb{Z}}$.  Informally, this means that the quantization error is pushed to high-frequency bands and subsequently suppressed by the low-pass filter $g$.    The following proposition provides rigorous bounds on the accuracy of $K$th order $\Sigma \Delta$ schemes; we refer the reader to \cite{DD03} for a detailed proof.
\begin{prop}
\label{standardstable}
Suppose $f  \in L^2(\mathbb{R})$ is bounded and has bandwidth $\Omega = 1/2$.  Let $f_n = f(\frac{n}{T})$ be  sampled at rate $T > 1$, and suppose $g$ in \eqref{shannonsamp} is chosen to have absolutely integrable $K$th derivative $\int_{-\infty}^{\infty} |g^{(K)}(t)|dt = \| g^{(K)} \|_{L^1} < \infty$.  Suppose $(v_n)_{n \in \mathbb{Z}}$ satisfies $|v_n | \leq v_{max}$ and is such that
\begin{equation}
\label{diff}
f_n - q_n = \Delta^K(v)_n
\end{equation}
for some sequence $(q_n)_{n \in \mathbb{Z}}$ with $|q_n| \leq 1$.  Here $\Delta^K(v)_n := \sum_{\ell=0}^K (-1)^{\ell} {K \choose\ell} v_{n-\ell}$ is the $K$th order difference operator.  Then the quantization error is uniformly bounded by
\begin{equation}
\label{standardacc}
\left| f(t) - \frac{1}{T} \sum_{n \in \mathbb{Z}} q_n g \Big(t - \frac{n}{T}\Big)  \right| \leq v_{max}\cdot  \| g^{(K)} \|_{L^1} \cdot  T^{-K}.
\end{equation}
\end{prop}
The error estimate \eqref{standardacc} guarantees on the reconstruction accuracy of a $K$th order sigma-delta scheme.  We note that more refined error estimates can be found, for example, in \cite{nguyen}. However, in practice the accuracy must be balanced against increased instability and implementation costs associated to higher-order schemes.  In practice it is common to use as low as second-order sigma-delta schemes.

In the sequel, we will restrict focus to second-order or double-loop $\Sigma \Delta$ schemes which generate $1$-bit quantization sequences $q_n \in \{-1,1\}$.  For a given input sequence $(f_n)_{n \in \mathbb{Z}}$, one method for generating output $(q_n)_{n \in \mathbb{Z}}$ and state sequence $(v_n)_{n \in \mathbb{Z}}$ to  satisfy the second-order finite difference equation $f_n - q_n = \Delta^2(v)_n$ is to start with $u_0 = 0$, $v_0 = 0$, and iterate for $n \geq 1$  
\begin{eqnarray}
q_n &=& {\cal Q}(u_{n-1} + \gamma v_{n-1} ) \nonumber \\
u_n &=&  u_{n-1} + f_n - q_n \nonumber \\
v_n &=&  u_{n-1} +  v_{n-1} + f_n - q_n.
\label{mainrecurs}
\end{eqnarray}
(the symmetric recursion can be implemented for $n \leq -1$).
Here, $\gamma$ is a parameter to be specified, and the quantizer ${\cal Q}: \mathbb{R} \rightarrow \{-1,1\}$ is the \emph{signum} function:
\begin{equation}
\label{signum}
{\cal Q}(x) := \left\{ \begin{array}{cc}
 1, & x \geq 0 \nonumber \\
-1, &x < 0,
\end{array} \right.
\nonumber
\end{equation}
Note that when $\gamma = 1$, \eqref{mainrecurs} gives the state equations for the typical double-loop $\Sigma \Delta$ modulator, as in \cite{motamed}. 
For the error guarantees in Proposition \ref{standardstable} to hold, we must ensure that the state sequence $(v_n)_{n \in \mathbb{Z}}$ remains bounded.   This is what is meant by \emph{stability} of a quantization scheme.  Treating the system \eqref{mainrecurs} as a non-stationary discrete dynamical system, Yilmaz proved in \cite{OYstability} that if $|f_n | \leq \beta < 1$, and if $\gamma$ is chosen within a certain viable $\beta$-dependent range, then stability is ensured and the state sequence can be bounded explicitly: $|v_n | \leq C_{\beta}$.  Improved bounds were later provided by Zeng \cite{zeng}.

 {\bf\normalsize Chaotic $\Sigma \Delta$ quantization.}
A persistent problem with the $\Sigma \Delta$ quantization scheme \eqref{mainrecurs}, and for higher-order schemes as well, is the production of \emph{periodic} output sequences $( q_n )_{n \in \mathbb{Z}}$.  In audio signal processing, where $\Sigma \Delta$ quantization is widely used, such periodicities can produce spikes in the frequency spectrum of the reconstructed signal $\widetilde{f}$, and manifest as audible idle tones to the listener.  Several attempts have been made to modify the standard recursion \eqref{mainrecurs} to break up periodic output without sacrificing accuracy of the resulting reconstruction.  One suggested approach is to apply dither, or white noise, to the input $f_n$ before implementing the recursion \cite{ChaosDither}.  Another approach, not necessarily mutually exclusive to dithering, is to break periodicities by amplifying the state sequence at each iteration, considering instead
\begin{eqnarray}
q_n &=& {\cal Q}(u_{n-1} + \gamma v_{n-1} ) \nonumber \\
u_n &=&  \lambda_1 u_{n-1} + f_n - q_n \nonumber \\
v_n &=&  \lambda_1 u_{n-1} +  \lambda_2 v_{n-1} + f_n - q_n;
\label{chaosrecurs}
\end{eqnarray}
This system is called \emph{chaotic} if either $\lambda_1 > 1$ or $\lambda_2 > 1$, following several simulation  studies \cite{SDChaos, ChaosDither, feelychaos, motamed} which indicate the effectiveness of this scheme for breaking up periodicities in the output $(q_n)_{n \in \mathbb{Z}}$.  As we will recall in the Appendix, second-order accuracy $|f(t) - \widetilde{f}(t) | \leq C_{v_{max}} T^{-2}$ of the standard scheme is maintained by \eqref{chaosrecurs} as long as $\lambda_1, \lambda_2$ are not too large compared to the sampling rate $T$, and as long as the state sequence $(v_n)$ remains bounded.  While stability for chaotic single-loop $\Sigma \Delta$ schemes ($K=1$) has been shown in \cite{feelychua92}, and stability for the double-loop chaotic schemes \eqref{chaosrecurs} was verified in \cite{feely98, motamed} for \emph{constant} input $f_n \equiv \beta < 1$, rigorous stability results for general bounded input $|f_n| \leq \beta < 1$ have been absent.  Stability depends on the choice of parameters $\lambda_1, \lambda_2, \gamma, $ and $\beta$;  as noted in \cite{feelychaos}, certain parameters lead to instability. 

{\bf\normalsize Contribution of this work.}  We give an explicit parameter range over which the second-order chaotic $\Sigma \Delta$ recursion is guaranteed to be stable, $|v_n| \leq v_{max}$, independent of the sampling rate. We moreover provide explicit bounds on $v_{max}$ - such bounds are crucial because quantization inaccuracies arise if state variables exceed certain device-dependent limits.
 
\section{Main results}
Consider the one-parameter family of second-order chaotic $1$-bit $\Sigma \Delta$ recursions
\begin{eqnarray}
q_n &=& {\cal Q}(u_{n-1} + \gamma v_{n-1} ) \nonumber \\
u_n &=&  \lambda u_{n-1} + f_n - q_n \nonumber \\
v_n &=&  \lambda u_{n-1} +  v_{n-1} + f_n - q_n,
\label{recurs2}
\end{eqnarray}
with quantizer ${\cal Q}$ returns the sign of its input as before.  Then the following stability result holds:

\begin{thm}
\label{thm1}
Fix $\lambda \geq 1$ and fix $0 < \alpha < 1$ so that 
\begin{equation}
\label{beta}
\beta := \frac{\alpha(1-\alpha) - 2\sqrt{2}(1+\alpha)(\lambda-1)}{1-\alpha + 2\sqrt{2}(1+\alpha)(\lambda-1)} > 0.
\end{equation}
Consider a sequence $(f_n)_{n \in \mathbb{Z}}$ uniformly bounded by $\sup_{n} | f_n | \leq \beta < 1$ as input to the chaotic double-loop $\Sigma \Delta$ recursion \eqref{recurs2} with multiplier 
\begin{equation}
\label{g}
\gamma = \frac{1-\alpha}{1+\alpha}
\end{equation} 
If $u_0 = v_0 = 0$, then the generated state sequence $(u_n, v_n)$ remains bounded and, in particular, $| v_n | \leq \frac{2(1+\alpha)}{1-\alpha} + \frac{1-\alpha}{8}$ for all $n \in \mathbb{Z}$.
\end{thm}
Note that when $\lambda = 1$, i.e. we implement the standard recursion, $\beta = \alpha$ and we recover the stability results of Yilmaz, \eqref{Yilmaz}.

\begin{figure}[h!]
\begin{center}
\includegraphics[width=6cm]{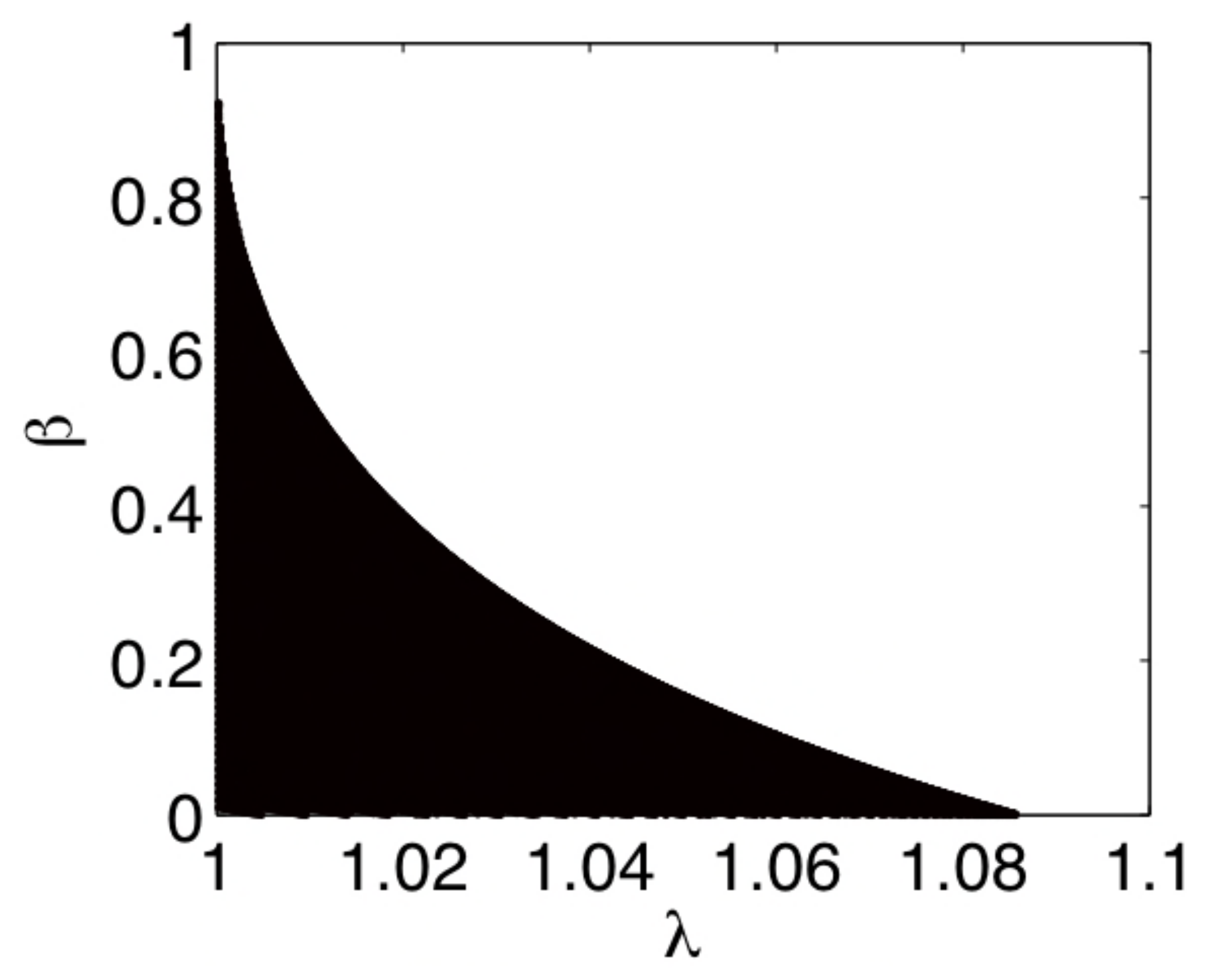}
\caption{Region of stability for the double-loop chaotic recursion \eqref{recurs2}: if the input is bounded by $\beta$ and implemented with parameter $\lambda \geq 1$, then the state sequence is guaranteed to be bounded as long as $(\lambda, \beta)$ falls under the curve and $\gamma$ is fixed according to Theorem \ref{thm1}.}
\label{fig:1}
\end{center}
\end{figure}

While the parameter $\gamma$ is assumed fixed in Theorem \ref{thm1}, we prove a more general stability result in Section \label{sec:proof} where $\gamma$ can assume any of an interval of values.  While the admissible range for $\beta$ as a function of $\lambda$ is not derived explicitly in Theorem \ref{thm1} , we numerically derive this range, as plotted in Figure \ref{fig:1}, and see for example that for positivity of $\beta$ in \eqref{beta}, necessarily $\lambda \leq 1.085$.   Smaller values of $\beta$ imply tighter bounds on the state sequence.

\section{Numerical experiments} 
In this section, we analyze the optimality of the bounds of Theorem \ref{thm1} via numerical simulations.  

\begin{itemize}

\item In Figure~\ref{fig:2}, we run \eqref{recurs2} with constant input sequence $ f_n \equiv \beta$ over a range of values of $\beta$, and over a range of values for $\lambda \geq 1$.  We let $\gamma$ be assigned according to Theorem \ref{thm1}.  The curve $\beta_{observed}$ traces the maximal observed value of $\beta$ for which the state sequence $(v_n)$ remains bounded.   By ``bounded", we mean that $| v_n | \leq 1000$ for 1 million iterations, a bound chosen from repeated observations.  We compare $\beta_{observed}$ to $\beta_{theoretical}$, the largest value of $\beta$ for which Theorem \ref{thm1} guarantees stability (in agreement with the curve in Figure \ref{fig:1}).  Clearly there is a gap between $\beta_{observed}$ and $\beta_{theoretical}$.  The true gap between experiment and theory may be smaller, though, because $\beta_{observed}$ traces the threshold for stability of constant input only.

 \begin{figure}[h!]
\begin{center}
\includegraphics[width=6cm, height=6cm]{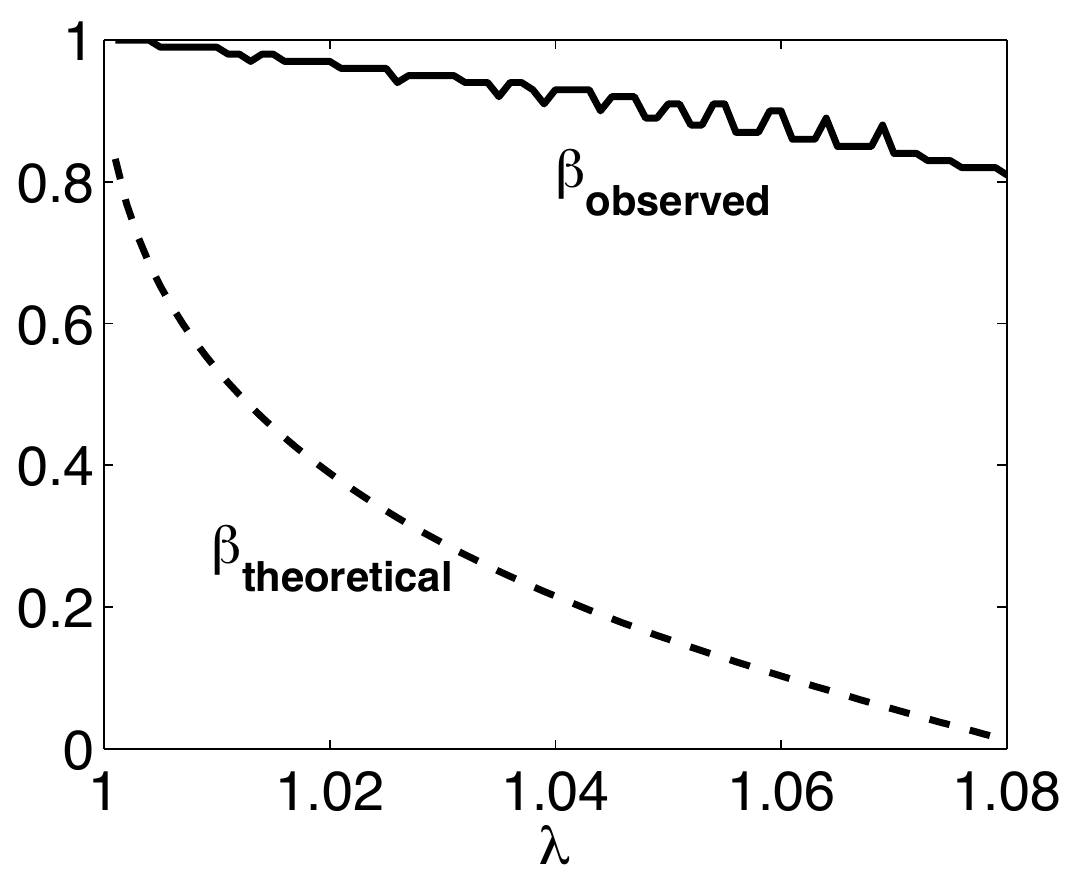}
\caption{Theoretical and empirical stability thresholds for the chaotic double-loop $\Sigma \Delta$ scheme as functions of the expansion $\lambda$ and level of constant input $f_n \equiv \beta$.  The parameter $\gamma = \gamma(\beta,\lambda)$ is as in Theorem \ref{thm1}.} 
\label{fig:2}
\end{center}
\end{figure}

\begin{figure}[h!]
\begin{center}
\includegraphics[width=7cm,height=6cm]{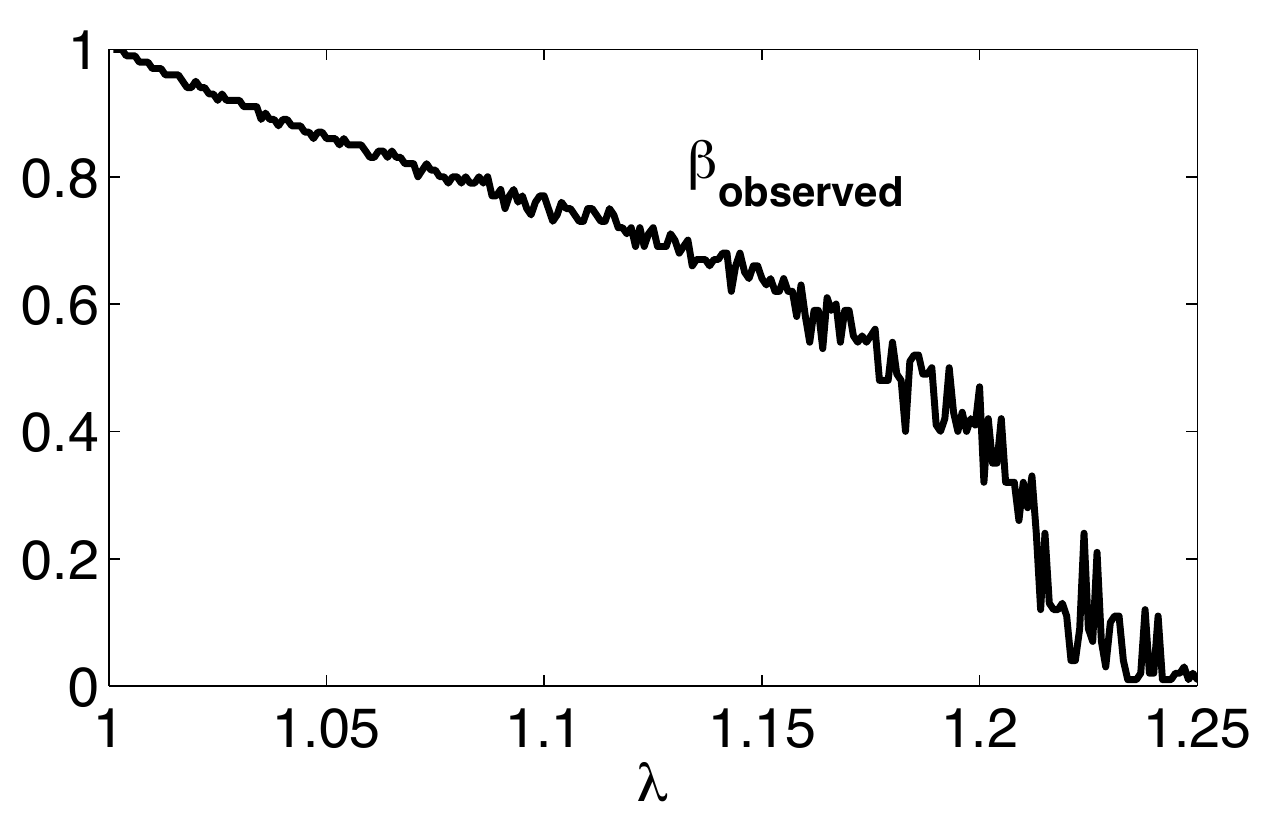}
\caption{Empirical stability threshold for the chaotic double-loop $\Sigma \Delta$ scheme as a function of the expansion $\lambda$ and level ofconstant input $f_n \equiv \beta$.  The parameter $\gamma = 1$ is fixed.}
\label{fig:3}
\end{center}
\end{figure}

\item In Figure \ref{fig:3} we repeat the procedure from Figure \ref{fig:2}, but now we fix $\gamma = 1$ and trace the minimal observed value of $\beta$ for which the state sequence $(v_n)$ fails to remain bounded.  We see that, at least for constant sequences $f_n \equiv \beta$, the chaotic $\Sigma \Delta$ recursion seems to be stable for a large range of $\gamma$ and over a larger range of $\lambda$ than that guaranteed by Theorem \ref{thm1}.

\begin{figure}[h!]
\begin{center}
\includegraphics[width=8cm, height=6cm]{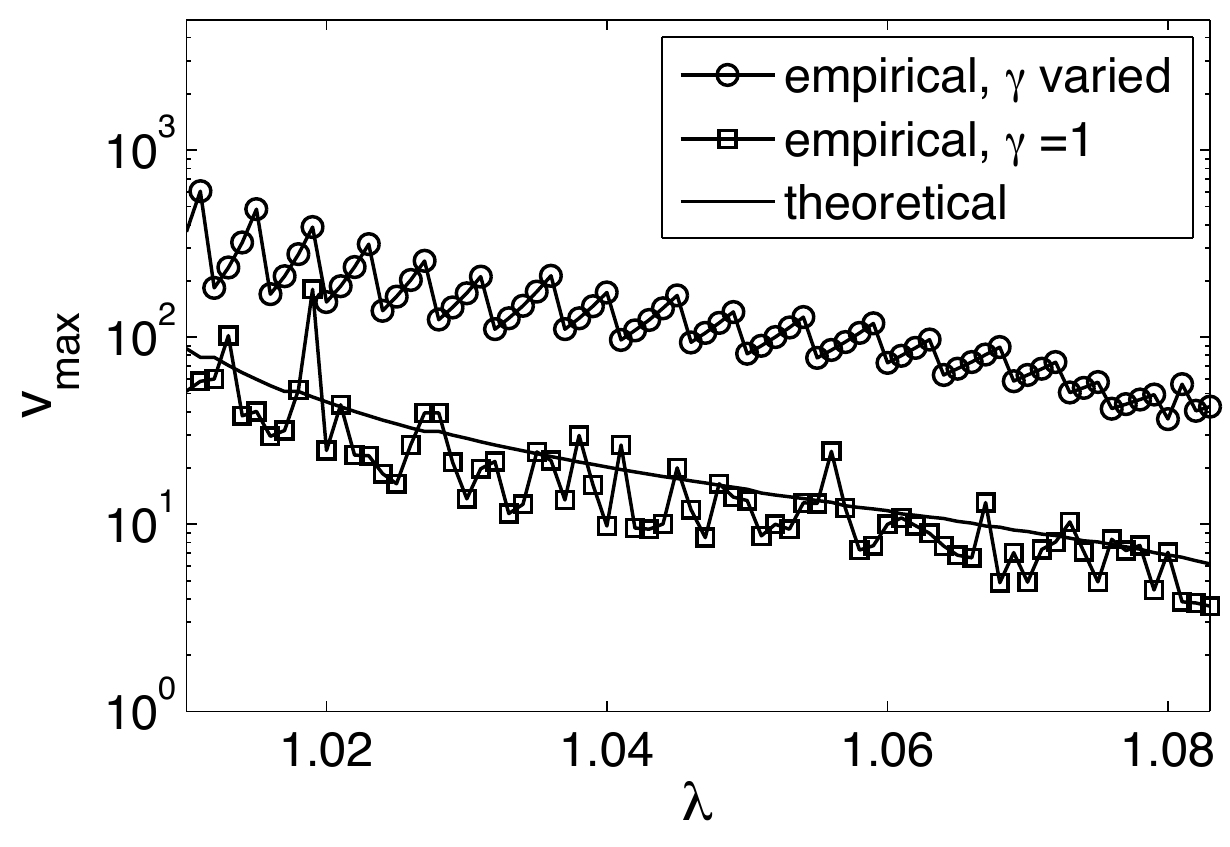}
\caption{Theoretical and empirical bounds on the state sequence $(v_n)$ for the chaotic double-loop $\Sigma \Delta$ scheme as functions of the expansion $\lambda$ and level of constant input $f_n \equiv \beta$. }
\label{fig:4}
\end{center}
\end{figure}

\item The set-up in Figure \ref{fig:4} is analogous to those of Figure \ref{fig:2} and Figure \ref{fig:3}, except we now compare the theoretical bound on the state sequence $v_{max} = \sup_n |v_n|$ given by Theorem \ref{thm1} to the maximal empirical bound on $|v_n|$ at the stability threshold $\beta_{observed}$ when $\gamma$ is varied as a function of $\beta$ and $\lambda$ according to Theorem \ref{thm1}. (open circles).  We also plot the maximal empirical bound on $|v_n|$ at the stability threshold corresponding to $\gamma \equiv 1$ (open squares).   The theoretical and empirical bounds nearly match, suggesting that although our bounds on the parameters $\beta$ and $\lambda$ might be conservative, our theoretical bounds on the state variable $v_{max}$ are near-optimal.  Note that there is no contradiction in observing larger values of $v_{max}$ than the theoretical upper bound given by Theorem \ref{thm1}, as the empirical bound is computed using a larger range for $\beta$.

\end{itemize}

\section{Set-up and Notation}

To prove Theorem \ref{thm1} we first introduce the notation on which the proof relies.  We consider input sequences $(f_n)_{n \in \mathbb{Z}}$ that are bounded $|f_n| \leq \beta < 1$.  We will derive a family of stable positively-invariant regions for the chaotic map \eqref{recurs2} as functions of the chaotic factor $\lambda$ and a parameter $\alpha \in [\beta, 1)$,  similar in spirit to the approach of \cite{OYstability}. We denote by $\delta_n := | f_n - q_n |$ the quantization error at time $n$.  Since $|f_n| \leq \beta \leq \alpha < 1$, this error will be confined to the interval $[1-\alpha, 1+\alpha]$.  We set $\delta_L = 1 - \alpha$ and $\delta_H = 1+\alpha$.  Consider 
\begin{equation}
\label{Sl}
S^{\delta}_{l}(u, v) = (u - \delta, u + v - \delta)
\end{equation}
and
\begin{equation}
\label{Sr}
S^{\delta}_r(u, v) = (u + \delta, u+ v + \delta).
\end{equation}
The update rule in \eqref{recurs2} can be rewritten as
\begin{eqnarray}
\label{S+}
(u_n, v_n) &=& \left\{  \begin{array}{ll} S^{\delta_n}_{l}(\lambda u_{n-1}, v_{n-1}) & \textrm{if } q_n =1, \\ \nonumber \\
S^{\delta_n}_r(\lambda u_{n-1}, v_{n-1})  & \textrm{if }q_n=-1;
\end{array} \right. 
\end{eqnarray}
we will use the shorthand notation
\begin{equation}
\label{S}
(u_n, v_n) = S(\lambda u_{n-1}, v_{n-1}).
\end{equation} 
Let us define the functions
\begin{eqnarray}
\label{b1}
B_1(u) &=& \left\{ \begin{array}{ll} \frac{-1}{2\delta_{H}}u^2 + \frac{1}{2}u + C, & u \leq 0, \nonumber \\ \nonumber \\
\frac{-1}{2\delta_{L}} u^2 + \frac{1}{2}u + C, & u > 0.
\end{array} \right. \nonumber \\ \nonumber \\
B_2(u) &=& \left\{ \begin{array}{ll} \frac{1}{2\delta_{L}}u^2 + \frac{1}{2}u - C, & u \leq 0, \nonumber \\ \nonumber \\
\frac{1}{2\delta_{H}} u^2 + \frac{1}{2}u - C, & u > 0,
\end{array} \right.
\end{eqnarray} 
where $C$ is a positive constant to be specified.  Note that $B_1$ and $B_2$ are convex and concave, respectively; thus the region ${R} = \{ (u,v): B_2(u) \leq v \leq B_1(u) \}$  is a convex set, and is illustrated in Figure \ref{fig:3}.  Consider now the action of $S$ from a point $(u_n, v_n)$ in ${R}$.  Depending on whether $\gamma v_n \geq -u_n$ or $\gamma v_n \leq -u_n,$ a left move $S_l$ or right $S_r$ will be applied to determine $(u_{n+1}, v_{n+1})$.  This suggests we split ${R}$ into two subsets, particularly:
\begin{eqnarray}
\label{SR}
{R}_1 &=& \{(u,v): \quad v \leq B_1(u), \quad v \geq B_2(u), \quad \gamma v \geq -u \} \nonumber \\
{R}_2 &=& \{(u,v): \quad v \leq B_1(u), \quad v \geq B_2(u), \quad \gamma v \leq -u\} \nonumber \\
{R} &=& {R}_1 \cup {R}_2.
\end{eqnarray}

\begin{figure}[h]
\begin{center}
\includegraphics[width=9cm]{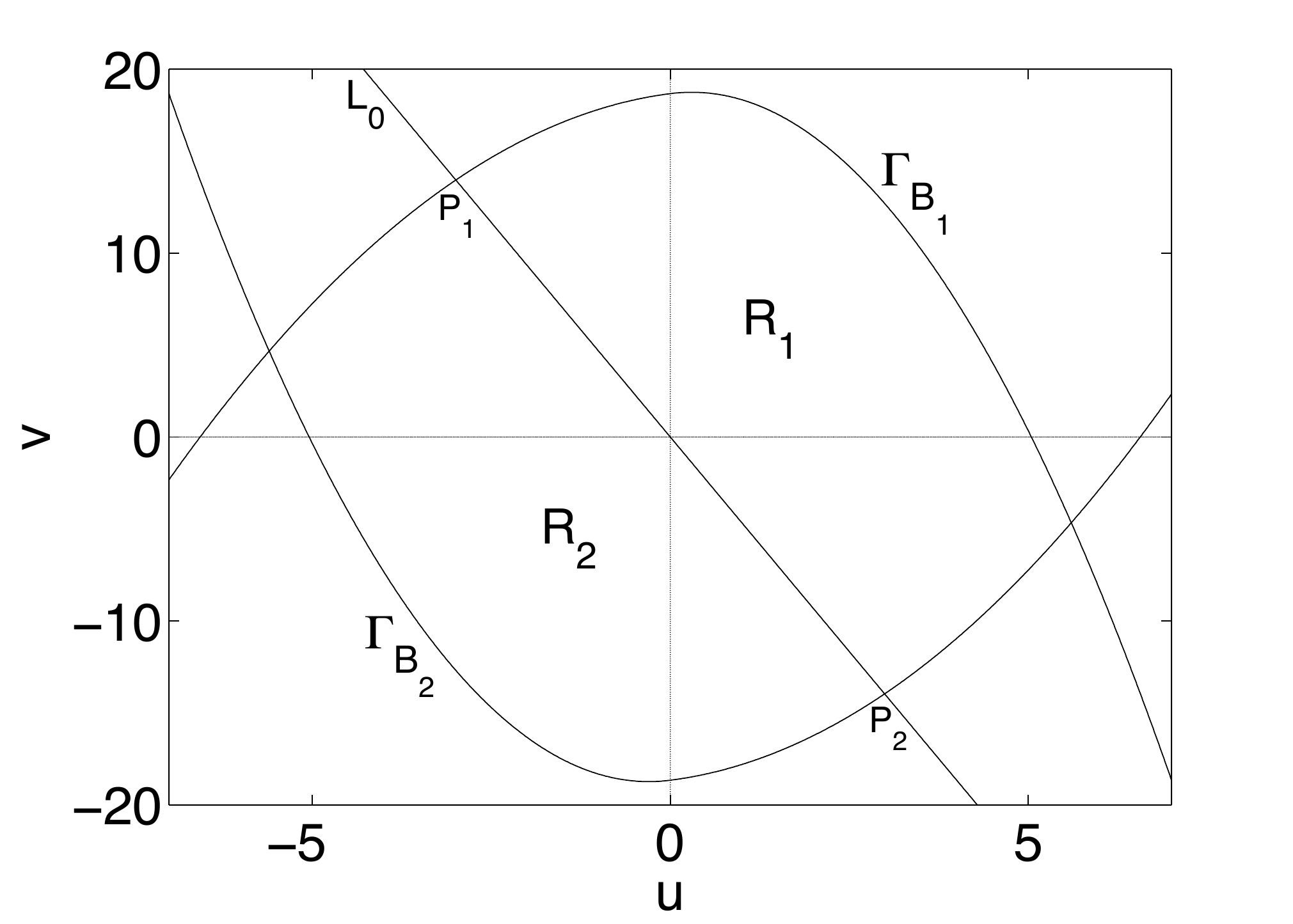}
\end{center}
\caption{An invariant set ${R} = R_1 \cup R_2$.}
\label{fig:5}
\end{figure}
As displayed in Figure $\ref{fig:3}$, the graphs of $B_1$ and  $B_2$ are denoted by $\Gamma_{B_1}$ and $\Gamma_{B_2}$ respectively.   We also illustrate the graph of the line $L_0: u + \gamma v = 0$. The left-most intersection point of $\Gamma_{B_1}$ and $\Gamma_{B_2}$ is denoted by $P_0 = (-u_0, v_0)$, while the intersection of $L_0$ and $\Gamma_{B_1}$ is denoted by $P_1=(-u_1,v_1)$, and the intersection of $L_0$ and $\Gamma_{B_2}$ by $P_2=(u_2,-v_2)$. Because ${R}$ is symmetric about the origin, we have that $(u_2,v_2)=(u_1, v_1)$.  A direct calculation yields $u_0 = (2C\delta_{H}\delta_{L} )^{1/2}.$
Another direct calculation gives
\begin{eqnarray}
\gamma = \frac{u_1}{\big(-\frac{1}{2\delta_{H}} u_1^2 - \frac{1}{2} u_1 + C \big)}.
\end{eqnarray}

{\section{Proof of stability for the chaotic quantization scheme}\label{sec:proof}} 

We now prove some supplementary results. It is through these supplementary results that we obtain the bounds on parameters $(\lambda, \beta, \gamma)$ that will guarantee the stability of the chaotic $\Sigma \Delta$ scheme.

In \cite{OYstability}, Yilmaz proved the following stability result for the map \eqref{S+} in the standard setting, when $\lambda = 1$.

\begin{prop}[Yilmaz]
\label{Yilmaz}
Consider the map \eqref{S+} with $\lambda = 1$.  Suppose that $|f_n| \leq \alpha < 1$.  Suppose that the constant $C$ in \eqref{b1} satisfies
\begin{equation}
\label{lowerc}
C \geq \frac{{2(1+\alpha)}}{1-\alpha}
\end{equation}
and, moreover, that the multiplier $\gamma$ lies in the range
\begin{equation}
\label{gamma}
\frac{1+\alpha}{C - (1+\alpha)} \leq \gamma \leq \frac{\sqrt{2C(1-\alpha)(1+\alpha)}-(1+\alpha)}{ C\alpha+\sqrt{C(1-\alpha)(1+\alpha)/2}}.
\end{equation}
\medskip
 Then $(u', v') = S(u,v) \in {R}$ for any $(u,v) \in {R}$.
\end{prop}
We observe that $C$, reflecting the size of the invariant region $R$, must blow up as $\alpha$ approaches $1$; in practice it is preferable to fix $\alpha = \alpha_{max}$ away from $1$ so that state variables remain sufficiently small.
We will leverage this result to prove stability of the map when $\lambda > 1$ by considering the same family of invariant sets \eqref{SR}, but restricting the dynamic range of the input to $|f_n| \leq \beta < \alpha$, where $\beta = \beta(\lambda)$ is diminished to counterbalance the expansion $\lambda$.  In particular, we fix a small parameter $\varepsilon > 0$ and take $\beta = \frac{\alpha-\varepsilon}{1+\varepsilon}$; note that $[1-\beta, 1+ \beta] \subset [\frac{1}{1-\varepsilon}\delta_L, \frac{1}{1+\varepsilon}\delta_H]$.

We first show that under slightly more restrictive conditions than those in Proposition \ref{Yilmaz}, $S_{l}(R_1)$ lies under the graph of $B_1$ even when $\lambda > 1$.

\begin{lem}
\label{invariant}
Fix $\alpha < 1$ and $\lambda > 1,$ and fix $\varepsilon \in (0,\alpha)$.
Suppose that $C$ and $\gamma$ satisfy  \eqref{lowerc} and \eqref{gamma}, and let $\delta_L = 1-\alpha$ and $\delta_H = 1 + \alpha$.  Suppose in addition that
\begin{equation}
\label{cbound}
C \leq \frac{\varepsilon^2 \hspace{1mm}\delta_{L}}{2(\lambda-1)^2\delta_{H}},
\end{equation}
or equivalently, that
\begin{equation}
\label{u0bound} 
u_0 \leq \frac{\varepsilon \delta_L}{\lambda-1}.
\end{equation}
If $(u,v) \in {R}$ and $(u', v') = S^{\delta}_l(\lambda u, v)$ for $\delta \in \big[\frac{1}{1 - \varepsilon}\delta_{L}, \frac{1}{1 + \varepsilon} \delta_{H}\big]$, then $v' \leq B_1(u')$.
\end{lem}

\begin{proof}
Suppose $(u,v) \in {R}$.  Let  $\delta \in [\frac{1}{1 - \varepsilon}\delta_{L}, \frac{1}{1 + \varepsilon} \delta_{H}]$ and consider $\delta' = \delta - (\lambda - 1)u$.  A straightforward computation shows that $S^{\delta}_l(\lambda u, v) = S^{\delta'}_l(u,v);$  
that is, a left-iteration by $\delta$ of the map with expansion $\lambda > 1$ is equal to a left-iteration by $\delta'$ of the standard map with $\lambda = 1$.  Since $C$ and $\gamma$ are admissible for the standard map by assumption, the stated result will follow by Proposition \ref{Yilmaz} as long as $\delta' \in [\delta_{L}, \delta_{H}]$.  Observe that because $(u,v) \in R$, $-u_0 \leq u \leq u_0$ and $v \leq B_1(u)$. Then $\delta' = \delta - (\lambda - 1)u \leq \delta_{H}$ for all $\delta \in [\frac{1}{1-\varepsilon} \delta_{L}, \frac{1}{1 + \varepsilon} \delta_{H}]$ and all $u \in [-u_0, u_0]$ if and only if $\frac{1}{1 + \varepsilon}\delta_{H} + (\lambda-1)u_0 \leq \delta_{H}.$
Rewe warranging, we see that this inequality is implied by the stated upper bound on $u_0$.
An analogous argument verifies further that $\delta' \geq \delta_{L}$; we leave the details to the reader.
\end{proof}

Lemma \ref{invariant} gives conditions under which $S_l(R_1)$ lies under the graph of $B_1$; to complete the argument that $S({R}) \subseteq {R}$, it remains to show that under a subset of the same conditions, $S_l({R}_1)$ lies above the graph of $B_2$.  Indeed, by symmetry of the action of $S$ and of the region $R$, $S_l(R_1) \subset R$ implies $S_r(R_2) \subset R$.   To show that $S_l(R_1)$ lies above the graph of $B_2$, we follow a similar line of argument to that used in \cite{OYstability} when $\lambda = 1$.

\begin{prop}
\label{prop4}
Fix $\alpha, C,$ and $\varepsilon$ as in Lemma \eqref{invariant} and consider the associated invariant set $R$ as defined in \eqref{SR}. Let $\delta_H = 1 + \alpha$ and $\delta_L = 1 - \alpha$.  Let $\beta = (\alpha - \varepsilon)/(1 + \varepsilon)$ and suppose that 
\begin{eqnarray}
\label{cond1} 
\delta_H  \leq u_1 \leq u_0 - \delta_H \label{cond1}.
\end{eqnarray} Then $S_l(R_1) \subseteq R$ for any $\delta \in [1-\beta, 1+\beta]$. 
\end{prop}
%Combining the bounds from \eqref{cond1} and \eqref{u0bound} gives the following bound on $u_1$:
%\begin{equation}
%\label{boundu1}
%u_1 \leq u_0 - \delta_H \leq \frac{\varepsilon \delta_L - (\lambda - 1)\delta_H}{\lambda-1}.
%\end{equation}

\begin{proof}
Lemma \eqref{invariant} established that $S_{l}(R_1)$ lies under $\Gamma_{B_1}$, therefore it is sufficient to show that $S_l(R_1)$ stays above $\Gamma_{B_2}$.  It is easily seen that if $ v_1 \geq v_2 $, then $(u, v_1)$ and $(u, v_2)$ get mapped to $(u' , v_1')$ and $(u' , v_2')$, with $ v_1' \geq v_2'$. Therefore, if we write $\Lambda = \{(u,v) : v = B_2(u)  \textrm{ and } u_2 \leq u \leq u_0 \}, $
 then we only need to ensure that the map of $\Lambda$, as well as that of the line segment connecting $P_1$ to $P_2$,  stays above $\Gamma_{B_2}$.
Additionally,  for the line segment it is sufficient to check just the end points $P_1$ and $P_2$ because the map $S_l$ is affine in $(u,v)$ due to the convexity of ${R}$. For each end point and for $\Lambda$, we only need to check $\delta_L' := 1-\beta$ and $\delta_H' := 1+\beta$ since the map $S_l$ is affine with respect to $\delta$.  
 
\begin{enumerate}

\item{\bf Case $P_1$:} Because $0< \delta'_L < \delta'_H$, and because the two points $\big( -(\lambda u_1 + \delta'_L), v_1 - (\lambda u_1 + \delta'_L) \big)$ and $\big( -(\lambda u_1+ \delta'_H), v_1 - (\lambda u_1 + \delta'_H) \big)$ lie on the same line with positive unit slope, whereas $B_2$ is decreasing for $u \leq 0$, both points lie above $\Gamma_{B_2}$ if the latter point is above $\Gamma_{B_2}$.  The condition \eqref{cond1} along with the upper bound \eqref{u0bound} imply:
\begin{eqnarray}
\label{u1array}
u_1 \leq u_0 \leq \frac{\varepsilon \delta_L}{\lambda-1} \leq \frac{\varepsilon \delta_H}{(\lambda-1)(1+\varepsilon)} = \frac{\delta_H - \delta'_H}{\lambda-1}.
\end{eqnarray}
 The third inequality follows because $\varepsilon < \alpha$.  Rearranged, this is $\lambda u_1 + \delta'_H \leq u_1 + \delta_{H},$ which combined with \eqref{cond1} implies that $ -u_0 \leq -(\lambda u_1 + \delta'_H)$.  \\

By construction, $B_1(u) = B_1(u - \delta_H) - (u-\delta_H)$. Thus
\begin{eqnarray}
v_1 - \lambda u_1 - \delta'_H = B_1(-u_1) - (\lambda u_1 + {\delta}'_H) && \nonumber \\
= B_1\big(-(u_1 + \delta_H) \big) + (u_1 + \delta_H) - (\lambda u_1 + \delta'_H) &&
\nonumber \\\geq B_1\big(-(u_1 + \delta_H) \big)
\geq B_2\big(-(u_1 + \delta_H) \big) &&\nonumber \\
\geq  B_2\big(-(\lambda u_1 + \delta'_H) \big) && 
\end{eqnarray}  
the third to last and last inequalities follow respectively from \eqref{u0bound} and the decrease of $B_2$ when $u \leq 0$.  

\item{ \bf Case $P_\Lambda$: }Recall that $P_2 = - P_1$ by symmetry of $B_1$ and $B_2$. Since $B_2$ is increasing for all $u>0$, it is sufficient to show that   $0 \leq \lambda u_\Lambda - \delta'_H \leq  \lambda u_\Lambda -\delta'_L \leq u_\Lambda$ for all points $P_\Lambda$.  The left-most inequality is true because $u_\Lambda \geq u_1 \geq \delta_H$ following \eqref{cond1}. All we have to show is that $\lambda u_\Lambda - \delta'_L \leq u_\Lambda$, or equivalently, that $u_\Lambda \leq u_0 \leq \delta'_L/(\lambda-1)$. But this is a consequence of \eqref{u0bound}.
\end{enumerate}
\end{proof}

With Proposition \ref{Yilmaz}, Proposition \ref{prop4} and Lemma \ref{invariant} in hand, we can now prove our main result.

\begin{thm}
\label{thm2}
Fix $\alpha < 1$ and $\lambda > 1$ such that $\lambda \leq 1 + \frac{\alpha(1-\alpha)}{2(1+\alpha)},$
and suppose that $\varepsilon$ is in the range 
$$\frac{2(1+\alpha)(\lambda-1)}{1-\alpha} \leq \varepsilon \leq  \alpha.$$
Let $\delta_H = 1+\alpha$ and $\delta_L = 1-\alpha$, and suppose $\gamma$ is in the range
$$
\frac{2\delta_H^2(\lambda-1)^2}{\varepsilon^2\delta_L - 2\delta_H^2(\lambda-1)^2} \leq \gamma \leq \frac{\varepsilon \delta_L^2 - \delta_H\delta_L(\lambda-1)}{\delta_H^2(\lambda-1)}.
$$
Consider the chaotic double-loop $\Sigma \Delta$ scheme \eqref{recurs2} with bounded input $\sup_{n} | f_n | \leq \beta < 1$, where $\beta = \frac{\alpha-\varepsilon}{1+\varepsilon}$.
If $u_0 = v_0 = 0$, then the state sequence $(u_n, v_n)$ remains bounded for all $n \in \mathbb{Z}$; in particular, $| v_n | \leq \frac{\varepsilon^2 \delta_L}{2\delta_H(\lambda-1)^2} + \frac{\delta_L}{8}$.
\end{thm}

\begin{proof}

\begin{enumerate}
 \item The bound $2\delta_H \leq u_0$ implied by \eqref{cond1} is equivalent to the condition $C \geq \frac{2\delta_H}{\delta_L}$,
 the lower bound \eqref{lowerc}. 
Combined with the upper bound \eqref{cbound}, 
%under which Proposition \ref{prop4} and Lemma \ref{invariant} hold:
 \begin{equation}
 \label{C2}
\frac{2\delta_H}{\delta_L} \leq C \leq \frac{\varepsilon^2 \delta_L}{2\delta_H(\lambda-1)^2}.
 \end{equation}  
The lower bound in this expression is smaller than the upper bound as long as
 \begin{equation}
 \label{eps2}
\frac{2\delta_H (\lambda - 1)}{\delta_L} \leq \varepsilon \leq \alpha,
 \end{equation} 
and this interval is nonempty only if
 \begin{equation}
 \label{lambda2}
 \lambda \leq 1 + \frac{\alpha(1-\alpha)}{2(1+\alpha)}.
 \end{equation}  
\item  Given $u_1$ and $C$, we can express $\gamma = u_1/B_1(-u_1)$ as
 \begin{equation}
 \label{gammaconst}
 \gamma = \frac{u_1}{C - (1/2)u_1 - 1/(2\delta_H) u_1^2}.
 \end{equation}
The range \eqref{cond1} for $u_1$ generates an admissible range for $\gamma$: 
 \begin{equation}
 \label{rangegamma}
\frac{\delta_H}{C - \delta_H} \leq \gamma \leq \frac{(2C\delta_H\delta_L)^{1/2} - \delta_H}{C\alpha+(C\delta_H\delta_L/2)^{1/2}}.
 \end{equation}
 \end{enumerate}
The bounds \eqref{C2} and \eqref{rangegamma} on $C$ and $\gamma$ respectively fall within the range of Proposition \ref{Yilmaz}, so we can apply Lemma \ref{invariant} as well as Proposition \ref{prop4} to conclude that if the parameters $(\alpha, \lambda, \varepsilon, C, \gamma)$ satisfy \eqref{C2}, \eqref{eps2}, \eqref{lambda2}, and \eqref{gammaconst}, and if $|f_n| \leq \beta = (\alpha-\varepsilon)/(1+\varepsilon)$, then $(u', v') = S(u,v) \in R$ if $(u,v) \in R$.  Theorem \ref{thm2} follows by maximizing the interval \eqref{rangegamma} for $\gamma$ using the bounds on $C$, and by noting that $\max_{(u,v) \in R} |v| \leq C + \frac{\delta_L}{8}$.
\end{proof}

\begin{rek}
Theorem \ref{thm1} is a special case of Theorem \ref{thm2}, obtained by setting $\varepsilon = \frac{2(1+\alpha)(\lambda-1)}{1-\alpha}$.
\end{rek}

\section{Discussion}
In this note we proved that second-order chaotic $\Sigma \Delta$ schemes \eqref{recurs2} are stable as long as $\lambda \geq 1$ is not too large and the dynamic range $\beta$ of the input is sufficiently small.  Our stability analysis can be extended to a more general setting and is presented in a limited scope for the sake of clarity.  For instance, our stability analysis also holds for tri-level quantizers in \eqref{recurs2} ; i.e. quantizers ${\cal Q}: \mathbb{R} \rightarrow \{-1,0,1\}$.  The motivation for using quantizers having a $0$-output state is to reduce power: in audio processing, over long stretches of input $f_n \approx 0$ (such as between songs on a recording), it is desirable that in response $q_n = 0$ so that the analog-to-digital converter can essentially ``shut off".  A so-called \emph{quiet} modification to the standard second-order scheme \eqref{mainrecurs} was introduced in \cite{ward2010} to have the property that if $f_n = 0$ over a sufficiently long stretch of time, then $q_n = 0$ in response.  This modification is at the same time guaranteed to retain second-order accuracy  of the standard scheme \eqref{standardacc}.   All of the stability analysis of this paper carries over to the quiet setting as well.

It remains to analyze the full two-parameter family of double-loop chaotic recursions \eqref{chaosrecurs} and to improve the stability theory to better match the empirical bounds (i.e., prove stability for a larger range of $\lambda$).  In addition to stability for this model, a rigorous study of the dynamics of the quantization output in the region of stability would be of interest.  In particular, it is still open whether the sequence of quantization output must be truly chaotic in the strict mathematical sense.

\section*{Acknowledgments}
We would like to thank Sinan G\"{u}nt\"{u}rk for valuable comments and improvements.  This work was supported in part by an NSF Mathematical Sciences Postdoctoral Reserach Fellowship.

\section{Appendix}
 In this appendix we show that for $\lambda_1$ and $\lambda_2$ sufficiently small, chaotic recursions of the form \eqref{chaosrecurs} really are second-order, in that they retain the second-order reconstruction accuracy with the sampling rate $T$ \eqref{standardacc} of the standard double-loop recursion.   The following proposition is adapted from \cite{OYstability}.

\begin{prop}
Let $f \in L^2(\mathbb{R})$, supp$(\widehat{f}) \in [-1/2, 1/2]$, $\| f \|_{L^{\infty}} \leq \beta < 1$, and consider samples $(f_n)_{n \in \mathbb{Z}} = \big(f(n/T)\big)_{n \in \mathbb{Z}}$ with $T > 1$.  Suppose that  $\lambda_1, \lambda_2$ in \eqref{recurs2} are such that $\lambda_1, \lambda_2 \leq 1 + 1/T$.  Suppose further that $|v_n| \leq v_{max}$.  Then
$$
\left| f(t) - \frac{1}{T} \sum_{n \in \mathbb{N}} q_n g\Big( t - \frac{n}{T} \Big) \right| \leq v_{max} \cdot C_g \cdot T^{-2}
$$
where $C_g$ is a constant depending on $g$, and $(q_n)_{n \in \mathbb{Z}}$ is the output produced from \eqref{recurs2}.
\end{prop}
\begin{proof}
Noting that $u_n = v_n - \lambda_2 v_{n-1}$,
\footnotesize{
\begin{eqnarray}
&& \left| f(t) - \frac{1}{T} \sum_{n \in \mathbb{Z}} q_n g\Big(t - \frac{n}{T}\Big) \right| =
\left| \frac{1}{T} \sum_{n \in \mathbb{Z}} (f_n - q_n) g\Big(t - \frac{n}{T}\Big) \right| \nonumber \\
&\leq& \left| \frac{1}{T} \sum_{n \in \mathbb{Z}} \Delta^2(v)_n g\Big(t - \frac{n}{T}\Big) \right| + (\lambda_1\lambda_2-1)\left| \frac{1}{T} \sum_{n \in \mathbb{Z}} \Delta^1(v)_{n-1} g\Big(t - \frac{n}{T}\Big) \right| \nonumber \\
&& \quad \quad \quad + (\lambda_1-1)(\lambda_2-1)\left| \frac{1}{T} \sum_{n \in \mathbb{Z}} v_{n-2} g\Big(t - \frac{n}{T}\Big) \right|  \nonumber 
\end{eqnarray}
}
\normalsize{
The first term is equivalent to the error of a stable second-order recursion $(\lambda_1 = \lambda_2 = 1)$ and is bounded by $T^{-2} \| v \|_{\infty} \| g^{(2)} \|_{L^1}$ by Proposition \ref{standardstable}.  The second term is similarly equivalent to the weighted error of  a standard first-order recursion;  since $| \lambda_1\lambda_2 - 1 | \leq 2/T$ it is bounded by $2T^{-2} \| v \|_{\infty} \| g^{(1)} \|_{L^1}$.  The final term is a `zeroth'-order recursion and is bounded by $T^{-2} \| v \|_{\infty} \| g \|_{L_1}$.}
\end{proof}

\bibliography{chaoticsd}
\bibliographystyle{abbrv}
\end{document}